\documentclass[11pt,reqno]{amsart}

\usepackage[T1]{fontenc}
\usepackage[utf8]{inputenc}
\usepackage{amsmath,amssymb,amsthm}
\usepackage{mathrsfs}
\usepackage[margin=1.15in]{geometry}
\usepackage[colorlinks=true,linkcolor=blue,citecolor=blue,urlcolor=blue]{hyperref}
\usepackage{graphicx}
\usepackage{float}

\theoremstyle{plain}
\newtheorem{theorem}{Theorem}[section]
\newtheorem{proposition}[theorem]{Proposition}
\newtheorem{lemma}[theorem]{Lemma}

\theoremstyle{definition}
\newtheorem{definition}[theorem]{Definition}

\newtheorem{example}[theorem]{Example}
\newtheorem{question}[theorem]{Question}

\theoremstyle{remark}
\newtheorem{remark}[theorem]{Remark}

\newcommand{\BB}{\mathbb{B}}
\newcommand{\RR}{\mathbb{R}}

\newcommand{\ZZ}{\mathbb{Z}}
\newcommand{\NN}{\mathbb{N}}

\newcommand{\Fn}{F_n}

\newcommand{\Homeo}{\operatorname{Homeo}}
\newcommand{\Conf}{\operatorname{Conf}}
\newcommand{\Aut}{\operatorname{Aut}}

\newcommand{\Fix}{\operatorname{Fix}}
\newcommand{\id}{\mathrm{id}}
\newcommand{\LB}{\mathrm{LB}}
\newcommand{\PLB}{\mathrm{PLB}}
\newcommand{\PHomeo}{\mathrm{PHomeo}}
\newcommand{\MCG}{\mathrm{MCG}}
\newcommand{\PMCG}{\mathrm{PMCG}}

\newcommand{\R}{\mathrm{R}}

\newcommand{\MM}{\mathrm{M}}
\newcommand{\N}{\mathrm{N}}

\newcommand{\Int}{\mathrm{Int}}
\newcommand{\CM}{\mathcal{M}}

\newcommand{\lk}{\operatorname{lk}}
\newcommand{\MF}{\operatorname{MF}}

\theoremstyle{plain}
\newtheorem{theoremx}{Theorem}

\newtheorem{propositionx}[theoremx]{Proposition}

\begin{document}
	
\title{Mixed configuration spaces, fixed points and Loop braid groups}
\author{Stavroula Makri}
\address{Department of Mathematics, Vrije Universiteit Amsterdam, Faculty of Science, De Boelelaan 1111, 1081 HV Amsterdam, the Netherlands}
\email{s.makri@vu.nl}
	
	\begin{abstract}
We introduce the notion of mixed configuration spaces of the $3$-ball. That is, we consider both distinct circles and distinct points in the interior of the $3$-ball and we prove that the projection onto the configuration space of circles of the $3$-ball is a locally trivial fibration. 

 The key idea for defining the mixed configuration spaces is to study fixed points  of orientation-preserving homeomorphism of the $3$-ball that leave invariant a trivial link of $n$ components in the interior of the $3$-ball. In particular, we prove that two fixed points are Nielsen equivalent if and only if the associated loop braids are conjugate by an element of a distinguished free subgroup of rank $n$. This result stands as a $3$-dimensional counterpart of the $2$-dimensional result where fixed points of homeomorphisms of the punctured disc are characterized in terms of braid elements of the classical Artin braid group.
\\\\
		\noindent \textit{Keywords:} Configuration spaces, Loop braid groups, Nielsen fixed point theory, Fixed points
	\end{abstract}
	\maketitle

	\section{Introduction}\label{intro}

In this work we introduce the notion of a mixed configuration space of the $3$-ball, that is, we consider distinct circles and points in the interior of the $3$-ball, which will be our main tool in studying fixed points of some homeomorphisms of the $3$-ball. Let $f$ be an orientation-preserving homeomorphism of the $3$-ball such that $f(C)=C$ for $C$ a trivial link of $n$ components in the interior of the $3$-ball. One is interested in determining and characterizing the Nielsen fixed point classes of such $f$. That is, to group together fixed points of $f$ that behave in a similar way. In \cite{Guaschi} Guaschi worked in a $2$-dimensional setting, in particular on the $n$-punctured disk, and provided necessary and sufficient conditions, in terms of braid elements, for two fixed points to be Nielsen equivalent. Consequently, the problem of determining whether two fixed points are Nielsen equivalent or not, which can be seen as a geometric and algebraic problem, is deduced to a conjugacy problem in braid group theory. Thus, the importance of this result lies in the fact that one can now use all tools of braid group theory but also knot invariants in order to study this problem. So far there has not been any similar development in a $3$-dimensional counterpart, mainly because the classical braid group theory becomes trivial in higher dimensions. In order to achieve an analogue $3$-dimensional development we introduce the notion of mixed configuration spaces and the use of loop braid groups.

The loop braid groups $\LB_n$ appear
in the literature under many different names. That is,  as basis-conjugating automorphisms and permutation
conjugacy automorphisms of $F_n$, due to  McCool and Savushkina \cite{McCool, Savushkina}, loop braid groups, due to Baez--Crans--Wise \cite{baez2007exotic}, groups of untwisted rings, due to Brendle--Hatcher \cite{BrendleHatcher} and welded braids, due to Fenn-Rimányi-Rourke \cite{FRR}. The group $\LB_n$
is in fact a $3$-dimensional analogue of the Artin braid group $\mathrm{B_n}$, where instead of
considering points on a surface we consider circles in the $3$-ball. There are several
interpretations of $\LB_n$; as the fundamental group of a
certain configuration spaces, as automorphisms of the free group $F_n$ and in terms of mapping class groups, as we will see in Section \ref{pre}. In terms of mapping class groups, the elements of loop braid groups are isotopy classes of orientation-preserving homeomorphisms $f$ of the $3$-ball that leave invariant a trivial link of $n$-components in the interior. This is a setting that has attracted a lot of attention and has been studied extensively from several viewpoints for various reasons \cite{Vershinin2001, damiani2017journey, Bellingeri2020, representationLoop, DarnePalmerSoulie2025, Makri2022}. Our main approach for studying the fixed points of such homeomorphisms $f$ in this $3$-dimensional setting is to define a mixed configuration space, where we consider circles and points inside the $3$-ball with the points playing the role of fixed points of $f$. We will prove that the projection of such a mixed configuration space onto the configuration space that contains only the circles, by forgetting the points, is actually a locally trivial fibration, which will be the main tool for providing necessary and sufficient conditions, in terms of loop braids, for two fixed points of $f$ to belong to the same Nielsen classes. We note that this mixed configuration space and the corresponding locally trivial fibration is inspired by the well-known Fadell--Neuwirth locally trivial fibration, where they consider the space of configurations of $n$ and $m$ distinct points and a projection map that "forgets" the last $m$ points. It is important to mention here that that the Fadell--Neuwirth locally trivial fibration \cite{FadellNeuwirth} has been one of the major tools in studying braid groups, see \cite{GuaschiJuanPineda2015}.

Let us present now the key ideas of this work, which is the notion of mixed configuration spaces and the existence of a projection map, which we prove that it is a locally trivial fibration. Let $\R_n$ denote the space of configurations of $n$ Euclidean, unordered, disjoint, unlinked
circles in the $3$-ball
lying on planes parallel to a fixed one. The fundamental group of $\R_n$ corresponds to the loop braid group, as explained in Section \ref{pre}. For $n, k\in \NN$ we define the $n,k$ \textit{mixed configuration space} to be the configuration space that consists of distinct circles and distinct ordered points. That is, $$\MM_{n,k}:=\{(c, p_1,\dots, p_k)\ |\ c\in \R_n,\ p_i\neq p_j\ \text{for}\ 1\leq i\neq j\leq k\ \text{and}\ p_i \in \Int(\BB^3)\setminus c\}.$$
We denote by $\CM_{n,k}$ the fundamental group of $\MM_{n,k}$. We consider the projection $p:\MM_{n,k}\to \R_n$ that forgets the $k$ points. Let $\Conf_n(X)$ denote the space of configurations of $n$ points in the space $X$. The theorem below appears in Section \ref{mix} as Theorem \ref{fibration}

\begin{theoremx}
The projection $p:\MM_{n,k}\to \R_n$ is a locally trivial fibration, where the fiber is $F:=\Conf_k(M)$, with $M=\Int(\BB^3)\setminus C$.
\end{theoremx}

Making use of the fact that $p$ is a locally trivial fibration we can use the long exact sequence on homotopy groups, from which we deduce a splitting short exact sequence, as described in the result that follows. In addition, this splitting will be an important factor in giving a characterization of fixed points in terms of loop braid groups, but before presenting both results, we need to fix some notation. Let $C=C_1\sqcup\cdots\sqcup C_n$ be a trivial link of unknotted circles in the interior of the
$3$-ball $\BB^3$ and let $f$ be an orientation-preserving homeomorphism of $\BB^3$ fixing $\partial\BB^3$ pointwise and preserving $C$. Let $M=\Int(\BB^3)\setminus C$, then $\pi_1(M)$ is isomorphic to a free group of rank $n$. As it is described in Section \ref{pre}, following $C$ along an
isotopy from the identity to $f$ yields a loop braid $\beta$ in the loop braid group $\LB_n$, and also $f$ induces an automorphism
$\phi_\beta$ of $F_n$. To a fixed point $x$ of $f$ in $M$ we associate three objects. Its coordinate $R(x)\in F_n$, which is the Reidemeister coordinate of Nielsen theory, obtained
by considering a path from a base-point to $x$ and its image under $f$. Its mixed motion
$l_x$, which records the trajectory of the circles together with the point $x$ during the isotopy. That is,
an element of the fundamental group $\CM_{n,1}$ of the mixed configuration space of $n$ circles and
one marked point, introduced in Section \ref{mix}. We note that $\CM_{n,1}$ has a free subgroup $U$ of
rank $n$, consisting of the motions in which the circles stand still and only the marked point
moves. Finally, thickening the marked point to a small circle turns $l_x$ into a loop braid
$t(l_x)$ on $n+1$ circles, and $U$ into the corresponding subgroup $U_{n+1}$ of $\LB_{n,1}$.
The following results appear in Section \ref{mix} and Section \ref{main} as Proposition \ref{ses} and Theorem \ref{thm2} respectively, where all constructions are presented in detail. 

\begin{propositionx}
     For every $k\geq 1$ we have the following short exact sequence:

    $$1\longrightarrow F_n^k\overset{\iota_*}{\longrightarrow}\CM_{n,k}\overset{p_*}{\longrightarrow}\LB_n\longrightarrow1.$$
    The short exact sequence splits, under the section $s_*$, and thus we have $\CM_{n,k}\cong F_n^k\rtimes\LB_n$, where $\LB_n$ acts on $F_n^k$ in the following way: for $\beta\in\LB_n$, whose induced element in $\Aut(F_n)$ is denoted by $\phi_\beta$, and for $(u_1,\dots,u_k)\in F_n^k$ we have $$s_*(\beta)^{-1}(u_1,\dots,u_k)s_*(\beta)=(\phi_\beta(u_1), \dots, \phi_\beta(u_k)).$$

    In particular, for $k=1$ we have $$\CM_{n,1}\cong U\rtimes\LB_n,\ \text{where}\ U\cong F_n.$$
\end{propositionx}

\begin{theoremx}
Let $x,y\in\Fix(f)\cap M$. The following are equivalent:
\begin{enumerate}
\item[(i)] the fixed points $x$ and $y$ are Nielsen equivalent in M, 
\item[(ii)] the coordinates $R(x)$ and R(y) are $\phi_\beta$-conjugate in $F_n$,
\item[(iii)] the mixed motions $l_x$ and $l_y$ are conjugate in $\CM_{n,1}$ by an element of $U$,
\item[(iv)] the loop braids $t(l_x)$ and $t(l_y)$ are conjugate in $\LB_{n,1}$ by an element of $U_{n+1}$.
\end{enumerate}
\end{theoremx}

Note that the equivalence between $\mathrm{(i)}$ and $\mathrm{(ii)}$ is a classical result in Nielsen theory, for which we give more information in Section \ref{pre}.
It is interesting to compare this result with the $2$-dimensional analogue given by Guaschi in \cite{Guaschi}. In his result he proves that two fixed points are Nielsen equivalent if and only if the corresponding braid elements are conjugate by an element in $U$, where $U$ is the kernel of the projection $\mathrm{pr}:\mathrm{B}_{n,1}\to\mathrm{B}_n$. By $\mathrm{B}_n$ we denote the classical
Artin braid group on $n$ strands, which is well-known that it is identified with the mapping class group of the disc relative to a
set $A$ of $n$ marked points, and $\mathrm{B}_{n,1}$ is the subgroup of $\mathrm{B}_{n+1}$ consisting
of braids whose underlying permutation fixes the last strand. Note that the map $\mathrm{pr}$ forgets
that last strand. However, the fundamental difference with our result is that in our setting two fixed points are Nielsen equivalent if and only if the corresponding loop braid elements are conjugate by an elements in $U_{n+1}$, where $U_{n+1}$ is a proper subgroup of the kernel of the projection $\mathrm{pr}: \mathrm{LB_{n,1}}\to \mathrm{LB_{n}}$ and it can not be the whole kernel. 

As a result, by introducing the notion of mixed configuration spaces, which is interested in its own right, we can completely characterize Nielsen classes of the homeomorphisms $f$ of $M$ using loop braids, translating the problem of Nielsen equivalence to a conjugacy problem in loop braid group theory. Thus, one can now use all tools of loop braid group theory to study this problem.

The structure of this paper is as follows: In Section \ref{pre} we provide the necessary background concerning Nielsen fixed point theory, as well as, loop braid group theory. In Section \ref{mix} we define the mixed configuration spaces, we prove that the map $p$ is a locally trivial fibration and we prove that the fundamental group of this mixed configuration space is isomorphic to a free group and a loop braid group. It is important to mention that this is the main ingredient for providing our characterization of fixed points of $f$ in terms of loop braid elements. In Section \ref{main} we give the characterization of Nielsen fixed point classes using loop braids. Finally, in Section \ref{application} we discuss outcomes of our constructions and results.

     \section{Prerequisites}\label{pre}
We start this section with an introduction in loop braid groups and we conclude it with a presentation in Nielsen fixed point theory.
    \subsection{Loop braid groups}
We will present three different but equivalent definitions of the loop braid groups $\LB_n$. For more insights on the equivalent definitions of $\LB_n$ we refer the reader to \cite{damiani2017journey}.

Let $M$ be a compact, connected, orientable manifold, possibly with boundary, and $N$ a submanifold contained in the interior of $M$. A self-homeomorphism of the pair $(M,N)$ is a
homeomorphism $f\colon M\to M$ that fixes $\partial M$ pointwise and fixes $N$. Note that $f$ induces
a permutation of the connected components of $N$. We denote by $\Homeo(M,N)$ the group of
self-homeomorphisms of $(M,N)$ preserving the orientations of both $M$ and $N$, with multiplication
given by composition, and by $\PHomeo(M,N)$ the subgroup of those sending each component of $N$ to
itself. Two elements of $\Homeo(M,N)$ are isotopic relative to $N$ if there is an isotopy between
them through elements of $\Homeo(M,N)$.

\begin{definition}
The \emph{mapping class group} of $M$ with respect to $N$, denoted $\MCG(M,N)$, is the group of
isotopy classes of elements of $\Homeo(M,N)$. The \emph{pure mapping class group} $\PMCG(M,N)$ is
the subgroup of classes sending each component of $N$ to itself.
\end{definition}

Fix $n\ge1$ and let $C=C_1\sqcup\cdots\sqcup C_n$ be a trivial link of $n$ disjoint, unknotted,
oriented circles lying in parallel planes in the interior of the $3$-ball $\BB^3$. The exact
position of $C$ is irrelevant and thus we assume that $C$ is contained in the $xy$-plane inside $\BB^3$. As already stated in the introduction, the loop braid groups are a three-dimensional analogue of the braid groups $\mathrm{B}_n$ defined by Artin in \cite{artin1947braids}. We know that $\mathrm{B}_n$ can be seen as $\MCG(\mathbb{D}^2,Q_n)$ for an $n$-point set $Q_n$, and similarly we have the following definition. 

\begin{definition}
The loop braid group on $n$ components is $\LB_n:=\MCG(\BB^3,C)$, and the \emph{pure loop
braid group} is $\PLB_n:=\PMCG(\BB^3,C)$.
\end{definition}

Thus, given $f\in\Homeo(\BB^3,C)$ and an isotopy $\{f_t\}_{t\in[0,1]}$ from
$\id$ to $f$, one obtains a loop braid by following $C$ under the isotopy. We denote this loop braid by $b_{f,C}\in\LB_n$. Note that such an isotopy exists since $\Homeo(\BB^3)$ is contractible by the Alexander trick. Moreover, every element of $\Homeo(\BB^3,C)$ is isotopic, through elements of $\Homeo(\BB^3,C)$, to one that is the identity on a collar of $\partial\BB^3$. This does not change neither the isotopy class nor $b_{f,C}$, and thus we can assume in the following sections that our homeomorphisms have this property.

The geometric model that we will also consider is the untwisted ring group of Brendle--Hatcher
\cite{BrendleHatcher}.

\begin{definition}
Let $\R_n$ be the space of configurations of $n$ disjoint, pairwise unlinked, unordered Euclidean
circles contained in $\Int(\BB^3)$, lying in planes parallel to a fixed plane $P_0$, and
let $\widetilde{\R}_n$ be the corresponding space of ordered configurations, so that
$\R_n=\widetilde{\R}_n/S_n$, where $S_n$ is the symmetric group of degree $n$. We set $\LB_n:=\pi_1(\R_n,C)$.
\end{definition}

We work throughout with Euclidean circles since this makes $\R_n$ a
finite-dimensional manifold and is what the proof of Theorem \ref{fibration} uses. However, this causes no issues because by \cite{BrendleHatcher} the space of configurations of Euclidean circles is homotopy
equivalent to the corresponding space of configurations of smoothly embedded circles. Thus,
$\LB_n$ and hence every statement that follows are not affected by considering topological circles instead of Euclidean ones.

\begin{theorem}[Brendle--Hatcher, \cite{BrendleHatcher}]\label{bre}
The group $\LB_n$ admits the presentation with generators
$\sigma_1,\dots,\sigma_{n-1},\rho_1,\dots,\rho_{n-1}$ and relations
\begin{enumerate}
\item $\sigma_i\sigma_{i+1}\sigma_i=\sigma_{i+1}\sigma_i\sigma_{i+1}$, for $i=1,\dots,n-2$,
\item $\sigma_i\sigma_j=\sigma_j\sigma_i$, for $|i-j|>1$,
\item $\rho_i\rho_{i+1}\rho_i=\rho_{i+1}\rho_i\rho_{i+1}$, for $i=1,\dots,n-2$,
\item $\rho_i\rho_j=\rho_j\rho_i$, for $|i-j|>1$,
\item $\rho_i^2=1$, for $i=1,\dots,n-1$,
\item $\sigma_i\rho_j=\rho_j\sigma_i$, for $|i-j|>1$,
\item $\sigma_i\rho_{i+1}\rho_i=\rho_{i+1}\rho_i\sigma_{i+1}$, for $i=1,\dots,n-2$,
\item $\rho_i\sigma_{i+1}\sigma_i=\sigma_{i+1}\sigma_i\rho_{i+1}$, for $i=1,\dots,n-2$.
\end{enumerate}
\end{theorem}

Relations $(1)$-$(2)$ are the braid relations, $(3)$-$(5)$ are the symmetric group relations, and $(6)$-$(8)$ are the mixed relations. The $\sigma_i$ generator permutes the $i^{th}$ and $(i+1)^{st}$ circles by passing the
$i^{th}$ through the $(i+1)^{st}$, and the $\rho_i$ generator permutes them by passing the $i^{th}$ around the
$(i+1)^{st}$, as in Figure \ref{figure}.

\begin{figure}[H]
	\centering
	\includegraphics[width=0.8\textwidth]{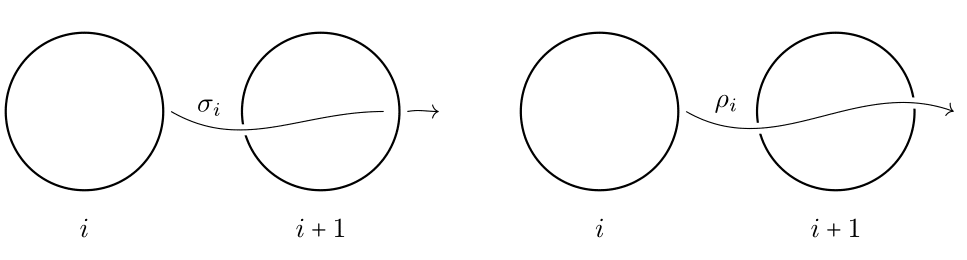}
	\caption{The generators $\sigma_i$ and $\rho_i$.}\label{figure}
\end{figure}

The loop braid group can be also considered in terms of automorphisms of the free group. 
Let $\Fn$ be the free group of rank $n$ with generators $x_1,\dots,x_n$. By
\cite{Goldsmith} the group $\LB_n$ embeds in $\Aut(\Fn)$, with image the group of permutation-conjugacy automorphisms. We note that the following result comes from Dahm’s unpublished thesis \cite{Dahm},
that appears in the last section of Goldsmith’s paper \cite{Goldsmith}. That is,
$$
a\colon x_i\longmapsto W_i^{-1}\,x_{\pi(i)}\,W_i,\qquad \pi\in S_n,\ W_i\in\Fn,
$$
and $\PLB_n$ corresponds to those with $\pi=\id$. Under this embedding the generators act as follows:

\begin{equation*}
\sigma_i:
\begin{cases}
x_i\mapsto x_i^{-1}x_{i+1}x_i\\
x_{i+1}\mapsto x_i\\
x_j\mapsto x_j,\quad j\neq i,i+1
\end{cases}
\qquad\text{and}\qquad
\rho_i:
\begin{cases}
x_i\mapsto x_{i+1}\\
x_{i+1}\mapsto x_i\\
x_j\mapsto x_j,\quad j\neq i,i+1.
\end{cases}
\end{equation*}

The $\sigma_i$ generate a copy of the braid group $\mathrm{B}_n$ and the $\rho_i$ a copy of the symmetric
group $S_n$. When we consider an element of $\LB_n$ as an automorphism we shall say so by writing
$\LB_n\subset\Aut(\Fn)$.

Our convention is that we use left conjugation, as in  \cite{BrendleHatcher,damiani2017journey}, instead of right conjugation, as in \cite{Makri, representationLoop}.

\subsection{Nielsen fixed point theory}
We will give an overview of Nielsen fixed point theory and for further insights in the theory we refer the reader to \cite{Jiang} and (\cite{JezierskiMarzantowicz2006}, Chapter IV).  
A classical problem related to Nielsen theory refers to determining the minimum number of fixed points among all maps homotopic to a given map $f$ from a compact space to itself.   
In particular, let $X$ be a connected finite cell complex, and $f:X\to X$ a continuous map. We denote by $\text{Fix}(f)$ the fixed point set of $f$, that is, $\text{Fix}(f)=\{x\in X\ | \ f(x)=x\}$. One is interested in studying the minimal number of the map $f$ denoted by $\MF(f)$ and defined as $\MF(f)=\text{min}\{\#Fix(g)\ | \ g\sim f\}$, where $\sim$ means homotopic. Two fixed points $x, y$ of $f$ are said to be \emph{Nielsen equivalent} if
there exists a path $\ell$ in $X$ from $x$ to $y$ such that $\ell$ and its image $f \circ \ell$ are homotopic relative to their endpoints. The equivalence class of a point $x \in \mathrm{Fix}(f)$ is called the \emph{Nielsen class} of $x$. We denote by $\mathrm{NC}(f)$ the set of Nielsen classes of $f$. A Nielsen class is \emph{essential} if it has nonzero fixed point index, as defined in (\cite{Jiang}, page 17) and the \emph{Nielsen number} $N(f)$ is defined as the number of essential Nielsen classes of $f$. The number $N(f)$ is a homotopy invariant which gives a lower bound for $\MF(f)$, that is $N(f)\leq \MF(f)$.

To compute $N(f)$ one needs to decide whether two fixed points belong to the same Nielsen fixed point class. This geometric problem was expressed in an algebraic context by Reidemeister in the following way. Choose a base-point $\ast_1 \in X$ and a path $w$ from $\ast_1$ to $f(\ast_1)$, called \emph{base path}. For simplicity, let $\pi = \pi_1(X, \ast_1)$ be the fundamental group of $X$ based at $\ast_1$. We define $f_\pi: \pi \to \pi$ as the composition of the induced map $f_*: \pi_1(X, \ast_1) \to \pi_1(X, f(\ast_1))$ with $w_*: \pi_1(X, f(\ast_1)) \to \pi_1(X, \ast_1)$. Two elements $\lambda_1, \lambda_2 \in \pi$ are said to be \emph{Reidemeister equivalent} or $f_\pi$-conjugate if there exists $\lambda \in \pi$ such that
\[
\lambda_2 = f_\pi(\lambda) \cdot \lambda_1 \cdot \lambda^{-1}.
\]
An equivalence class of elements of $\pi$ under this relation is called a \emph{Reidemeister class} and we denote by $ \mathcal{R}(f_\pi)$ the set of all Reidemeister classes. 

We are ready to assign a Reidemeister class to a fixed point of $f$. For $x \in \mathrm{Fix}(f)$, choose a path $\ell$ from $\ast_1$ to $x$. The \emph{Reidemeister class} or \emph{coordinate} of $x$ is the class $[w(f \circ \ell)\ell^{-1}]\in \pi$ and will be denoted by $R(x)$. Note that the Reidemeister class represented by $[w(f \circ \ell)\ell^{-1}]$
is independent of the choice of $\ell$, since for another path $\ell'$ from $\ast_1$ to $x$, the loops $w(f \circ \ell)\ell^{-1}$ and $w(f \circ \ell')(\ell')^{-1}$  have the same Reidemeister class, as we will see in the following Lemma. All points in a Nielsen class $K$ share the same Reidemeister class, which is denoted by $R(K)$ and called the Reidemeister class of $K$ or \emph{coordinate} of $K$.

From now on we assume that the base-point $\ast_1$ is a fixed point of $f$, which will be the case in Section \ref{mix}, where $f$ is the identity on the collar of $\partial\BB^3$. We can then take the base path $w$ to be constant, so that $f_\pi=f_*$ and the coordinate of a fixed point is simply $R(x)=(f\circ\ell)\ell^{-1}$. The following result puts together two important facts that we will use in Section \ref{main}.

\begin{lemma}\label{Nielsen}
Let $x,y\in\Fix(f)$ and $R(x)=(f\circ \ell)\ell^{-1}\in \pi$. Then:
\begin{enumerate}
\item[(i)] Let $\gamma\in\pi$. Replacing $\ell$ by $\gamma\ell$,
replaces $R(x)$ by $f_{\pi}(\gamma)R(x)\gamma^{-1}$. Thus, the class of $R(x)$ in $\mathcal{R}(f_\pi)$
is independent of the choice of $\ell$.
\item[(ii)] The fixed points $x$ and $y$ are Nielsen equivalent in $X$ if and only if $R(x)$ and $R(y)$ are Reidemeister equivalent.
\end{enumerate}
\end{lemma}

\begin{proof}
(i) We have $f\circ(\gamma\ell)=(f\circ\gamma)(f\circ\ell)$ and
$(\gamma\ell)^{-1}=\ell^{-1}\gamma^{-1}$. Thus, we get
$(f\circ(\gamma\ell))(\gamma\ell)^{-1}=(f\circ\gamma)(f\circ\ell)\ell^{-1}\gamma^{-1}$. It follows that these two loops are Reidemeister equivalent and thus they represent the same Reidemeister class.

(ii) Suppose $l$ is a path from $x$ to $y$ such that $l\simeq f\circ\ell$ keeping the endpoints fixed. We choose a path $\ell_x$ from $\ast_1$ to $x$ and we choose the path $\ell_y$ from $\ast_1$ to $y$ to be $\ell_xl$. Then
$
(f\circ\ell_y)\ell_y^{-1}=(f\circ\ell_x)(f\circ l)l^{-1}\ell_x^{-1}$. But $(f\circ\ell_x)(f\circ l)l^{-1}\ell_x^{-1}\simeq (f\circ\ell_x)\ell_x^{-1}$. Thus, $R(y)=R(x)$. Conversely, if
$R(y)=f_{\pi}(\gamma)R(x)\gamma^{-1}$, for $\gamma\in \pi$, then by replacing $\ell_x$ by $\gamma^{-1}\ell_x$ we get that $R(x)=R(y)$ and hence $(f\circ\ell_x)\ell_x^{-1}=(f\circ\ell_y)\ell_y^{-1}$. Hence, we can define the path
$l:=\ell_x^{-1}\ell_y$ from $x$ to $y$ that satisfies the relation $l\simeq f\circ l$ while keeping the
endpoints fixed.
\end{proof}

From the above result we know that two fixed points belong to the same Nielsen class if and only if they have the same Reidemeister class. Thus, $\mathrm{NC}(f)$ can be viewed as a subset of $ \mathcal{R}(f_\pi)$. For $\alpha \in \mathcal{R}(f_\pi)$, define $\mathrm{Fix}_\alpha(f) = \{x \in \mathrm{Fix}(f) \ | \ R(x) = \alpha\}$. Then, $$\text{Fix}(f)=\bigcup_{\alpha\in \mathcal{R}(f_\pi)} \mathrm{Fix}_\alpha(f).$$
If $\mathrm{Fix}_\alpha(f)$ is non-empty, it forms a Nielsen class with Reidemeister class $\alpha$. Note that the compactness of $X$ implies that $\mathrm{Fix}_\alpha(f)$ is empty except for finitely many $\alpha$.

\section{Mixed Configuration spaces}\label{mix}

     In this section we introduce the notion of mixed configuration space. Throughout, $f$ denotes an element of $\Homeo(\BB^3,C)$ which is the identity on the collar of
$\partial\BB^3$, and $\{h_t\}$ an isotopy from $\id$ to $f$ through such homeomorphisms. We denote by
$\beta:=b_{f,C}\in\LB_n$ the induced loop braid. We write $\phi_\beta$ for the automorphism of $\pi_1(M)\cong\Fn$ induced by $\beta$, which is denoted as $f_\pi$ in Section \ref{pre}. 
Moreover, we denote by $M$ the space $\Int(\BB^3)\setminus C$, which is homotopy equivalent to a wedge of $n$ circles and $n$
spheres. It follows that $\pi_1(M)$ is isomorphic to the free group of rank $n$. That is, $\pi_1(M)\cong F_n$ generated by $\{x_1,\dots,x_n\}$, where $x_i$ corresponds to the meridian of $C_i$. The reason behind considering mixed configurations spaces is to study fixed points of $f$.

     \begin{definition}
     For $n, k\in \NN$ we define the $n,k$ \textit{mixed configuration space} to be the configuration space that consists of distinct circles and distinct ordered points. That is, $$\MM_{n,k}:=\{(c, p_1,\dots, p_k)\ |\ c\in \R_n,\ p_i\neq p_j\ \text{for}\ 1\leq i\neq j\leq k\ \text{and}\ p_i \in \Int(\BB^3)\setminus c\}.$$
         
     \end{definition}

The space $\MM_{n,k}$ is a submanifold of $\R_n\times \Int(\BB^3)^k$ and is of dimension $4n+3k$. We fix a trivial link $C=C_1\sqcup \cdots \sqcup C_n$ of disjoint unknotted circles in parallel planes inside $\BB^3$ and $k$ distinct points $*_1,\dots,*_k$ in the collar of the boundary of $\BB^3$ that will be used as base-points. We denote by $\CM_{n,k}$ the fundamental group of $\MM_{n,k}$. That is, $\CM_{n,k}:=\pi_1(\MM_{n,k}, (C, *_1,\dots,*_k))$.

We consider the projection $p:\MM_{n,k}\to \R_n$ that forgets the $k$ points. We note that we can define a continuous section $s$ of $p$. The map $s: \MM_n\to \MM_{n,k}$ defined by $s(c):=(c,*_1,\dots,*_k)$ is a continuous section of $p$ since the circles $c$ lie in $\Int(\BB^3)$ and the marked points $*_1,\dots,*_k$ lie in the collar of $\BB^3$. In the following theorem we prove that the map $p$ is actually a locally trivial fibration.  
     
\begin{theorem}\label{fibration}
The projection $p:\MM_{n,k}\to \R_n$ is a locally trivial fibration, with fiber $F:=\Conf_k(M)$.
\end{theorem}

\begin{proof}
We want to prove that for every $c\in \R_n$ there exists a neighborhood $V$ of $c$ and a homeomorphism $\phi: V\times F\to p^{-1}(V)$, with $p\circ\phi=pr_V$, where $pr_V:V\times F\to V$ is the projection to $V$. From this it follows that $p^{-1}(c)\cong F$ for every $c\in \R_n$.

Suppose that for some neighborhood $V$ of $c_0\in \R_n$ we have, for each $c\in V$, a homeomorphism $\Phi_c$ of $\BB^3$, such that $\Phi_c(c_0)=c$ as subsets of $\BB^3$ and such that the map from $V$ to $\Homeo(\BB^3)$ that takes $c$ to $\Phi_c$ is continuous. Let $F_0:=\Conf_k(\Int(\BB^3)\setminus c_0)$. We can then define the map $\phi_0: V\times F_0\to p^{-1}(V)$ given by $$\phi_0(c,y_1,\dots,y_k)=(c,\Phi_c(y_1),\dots,\Phi_c(y_k)).$$
We note that the image of $V\times F_0$ under $\phi_0$ is indeed in $p^{-1}(V)\subset\MM_{n,k}$. For $y_i\notin c_0$ then $\Phi_c(y_i)\notin \Phi_c(c_0)=c$, since $\Phi_c$ is injective. Thus, the points $\Phi_c(y_i)$ are pairwise distinct and do not lie on $c$. Moreover, $\phi_0$ is a homeomorphism with inverse
$\phi_0^{-1}(c,y_1,\dots,y_k)=(c,\Phi^{-1}_c(y_1),\dots,\Phi^{-1}_c(y_k))$ and $p\circ\phi_0=pr_V$. Indeed, the
evaluation map $\mathrm{ev}:\Homeo(\BB^3)\times\BB^3\to\BB^3$, given by $\mathrm{ev}(\Psi,y):=\Psi(y)$, is continuous for the
compact-open topology, $\BB^3$ being compact, so each coordinate 
of $\phi_0$ is continuous. Since inversion is also continuous on $\Homeo(\BB^3)$, the same argument applies to $\phi_0^{-1}$. We conclude that $\phi_0$ is a trivialization of $p$ over $V$ with fiber $F_0$. In addition, note that $\R_n$ is path connected, since an appropriate shrinking of the radii followed by a translation of the centers joins any configuration to a standard one, the circles remaining disjoint throughout. So isotopy extension provides a homeomorphism $h$ of $\BB^3$ with $h(C)=c_0$ from which we can deduce a homeomorphism $\tilde{h}$ from $F$ to $F_0$, where $(y_1,\dots,y_k)$ is mapped to $(h(y_1),\dots,h(y_k))$. As a result, we can construct the desired homeomorphism $\phi:=\phi_0\circ(\id_V\times\tilde{h})$  from $V\times F\to p^{-1}(V)$ such that $p\circ\phi=pr_V$. Therefore, it suffices to construct the family of homeomorphisms $\{\Phi_c\}_{c\in V}$ of $\BB^3$.

Fix an ordering $c_0=(\rho^0_1,\dots,\rho^0_n)$ of the circles of $c_0$, that is, a preimage of
$c_0$ under the covering $\widetilde{\R}_n\to\R_n$. Since the map is a covering, there is a neighborhood
$V_1$ of $c_0$ and a unique continuous way of ordering the circles of every $c\in V_1$ that agrees
with the chosen ordering at $c=c_0$. For $c\in V_1$ we may therefore write
$c=(\rho_1,\dots,\rho_n)$. Note that each $\rho_i^0$ is a circle determined by its center $z_i^0$ and its radius $r_i^0$. That is, $\rho_i^0=(z_i^0, r_i^0)$. 
We choose pairwise disjoint compact balls $T_1,\dots,T_n$ in the interior of $\frac{3}{4}\BB^3$, such that $\rho_i^0$ lives in the interior of $T_i$. Moreover, we choose bump functions $\chi_i:\RR^3\to[0,1]$ supported in $T_i$ and equal to $1$ on a neighborhood of $\rho_i^0$. In addition, for any $c=(\rho_1,\dots,\rho_n)\in V_1$ we define the affine transformation $A_i^c$, which sends $\rho_i^0$ of $c_0$ onto $\rho_i$ of $c$. That is, $A_i^c:\RR^3\to\RR^3$ is given by $$A_i^c(y):=z_i+\frac{r_i}{r_i^0}(y-z_i^0),$$ which is actually a dilation followed by a translation that carries $(z_i^0, r_i^0)$ to $(z_i, r_i)$. Note that $A_i^c$ has no rotational part and it maps each plane to a parallel plane. Finally, for each $c\in V_1$ we define $\Phi_c:\BB^3\to\BB^3$ by $$\Phi_c(y):=y+\sum_{i=1}^n\chi_i(y)(A_i^c(y)-y).$$

By construction, $\Phi_c$ is continuous, $\Phi_{c_0}=\id$, $\Phi_c$ is the identity outside the set $\bigcup_iT_i$ and thus also near $\partial\BB^3$. Since on a neighborhood of $\rho^0_i$ the map $\chi_i$ equals $1$, then $\Phi_c$ agrees with $A^c_i$. As a result we have $\Phi_c(\rho^0_i)=A^c_i(\rho^0_i)=\rho_i$ for each $1\leq i\leq n$. That is, $\Phi_c(c_0)=c$ as subsets of $\BB^3$.

In what follows we will prove that $\Phi_c$ is a local diffeomorphism, injective and surjective. As $c$ tends to $c_0$, the maps $A_i^c$ converge to the identity in $C^1$ on the compact set
$\bigcup_iT_i$, since $A^c_i(y)-y=(z_i-z^0_i)+\left(\frac{r_i}{r^0_i}-1\right)(y-z^0_i)$ has both its
values and its derivative tending to $0$. As the $\chi_i$ are fixed, it follows that
$\|D\Phi_c-I\|$ tends to $0$ as well, so there is a neighborhood $V:=V_2\subseteq V_1$ of $c_0$ and a constant
$\lambda<1$ with $\|D\Phi_c-I\|\le\lambda$ for all $c\in V$. Thus, $\Phi_c$ has invertible derivative and so it is a local diffeomorphism. Moreover, $\Phi_c$ is injective. Since $\|D\Phi_c-I\|\le\lambda$, by the mean value inequality, it follows that $|(\Phi_c(y)-y)-(\Phi_c(y')-y')|\leq\lambda|y-y'|$ and therefore
$
|\Phi_c(y)-\Phi_c(y')|\ \ge\ |y-y'|-\lambda\,|y-y'|\ =\ (1-\lambda)\,|y-y'| .
$
Since $\lambda<1$, the equality $\Phi_c(y)=\Phi_c(y')$ forces $y=y'$. Finally, $\Phi_c$ is surjective. The image of $\BB^3$ under $\Phi_c$ is both open and closed. It is open because the points of $\partial\BB^3$ are fixed by
$\Phi_c$, and at interior points $\Phi_c$ is a local homeomorphism, hence an open map. It is also
closed, because it is the image of the compact set $\BB^3$ under the continuous map $\Phi_c$. As $\BB^3$ is connected and the image is
nonempty, $\Phi_c(\BB^3)=\BB^3$.

All together, we deduce that $\Phi_c$ is a diffeomorphism of $\BB^3$, since it is a bijective local diffeomorphism. In
particular it is a homeomorphism of $\BB^3$, which is all we need. Moreover, the map from $V$ to $\Homeo(\BB^3)$ that takes $c$ to $\Phi_c$ is continuous since $\Phi_c(y)$ depends continuously on $(c,y)$ and $\BB^3$ is compact. This is the required family
$\{\Phi_c\}_{c\in V}$ and this completes the proof.
\end{proof}

This result can be considered as an analogue of the Fadell--Neuwirth fibration for configuration spaces of points
in a manifold \cite{FadellNeuwirth}, which is one of the basic tools in the study of surface braid groups, see \cite{GuaschiJuanPineda2015}.

\begin{remark}\label{section}
The marked base-points $\ast_1,\dots,\ast_k$ in the collar of $\partial\BB^3$ are chosen once and for all and are
independent of the configuration in $\MM_{n,k}$. Thus, the map
$s\colon \R_n\longrightarrow\MM_{n,k}$ given by $s(c)=(c,\ast_1,\dots,\ast_k),$
is continuous and $p\circ s=\id$. Hence, $s$ is a section for the map $p: \MM_{n,k}\to \R_n$.
\end{remark}

In the following lemma we compute the fundamental group of the fiber of $p$, which we will use for proving in Proposition \ref{ses} the existence of a short exact sequence.

\begin{lemma}\label{free}
 For $k\geq1$, the following isomorphism hold: $$\pi_1(\Conf_k(M))\cong \pi_1(M)^k.$$
 In particular, $\pi_1(\Conf_k(M))\cong\Fn^k$, where $\Fn$ denotes the free group of rank $n$. 
\end{lemma}

\begin{proof}
We will use transversality to prove the claim. The fat diagonal $\Delta$ of $M^k$, that is the subspace of a $M^k$ where at least two of the coordinate components are equal, is a finite union of submanifolds of dimension $3(k-1)$. The inclusion $i$ of $M^k\setminus\Delta$ into $M^k$ induces a surjective map $i^*:\pi_1(\Conf_k(M))=\pi_1(M^k\setminus\Delta)\to \pi_1(M^k)$. This holds because, by transversality, every map of a $1$-complex into $M^k$ may be perturbed away from the fat diagonal $\Delta$, since $1+3(k-1)<3k$. Moreover, the induced map $i^*$ is injective because, again by transversality, every map of a $2$-complex, which corresponds to null-homotopies, into $M^k$ can be perturbed away from $\Delta$, since $2+3(k-1)<3k$. Therefore, $\pi_1(\Conf_k(M))\cong \pi_1(M^k)$, where the isomorphism is induced by the coordinate projection. Note that $\pi_1(M^k)=\pi_1(M)^k$ and since $\pi_1(M)=\Fn$ it follows that $\pi_1(\Conf_k(M))\cong\Fn^k$. 
\end{proof}

 Let $\R_{n,1}$ be the space of
configurations of $n+1$ such circles with the last one distinguished, and set
$\LB_{n,1}:=\pi_1(\R_{n,1})$. Equivalently, $\LB_{n,1}$ is the subgroup of $\LB_{n+1}$ consisting of
the elements whose underlying permutation fixes the last component. We denote by
$q\colon\LB_{n,1}\to\LB_n$ the homomorphism that forgets the last circle. In the following result we show that we can actually replace the point of the configuration space $\MM_{n,1}$ with a circle and obtain an element in $\R_{n,1}$. In this way we can associate an element of $\LB_{n,1}$ to an element in $\CM_{n,1}$, which we will use in the upcoming Lemma \ref{correspondance}

\begin{lemma}\label{thickening}
There is a homomorphism $t:\CM_{n,1}\to \LB_{n,1},$
that satisfies $q\circ t=p_*$, where $p_*:\CM_{n,1}\to \LB_n$, the map induced by the projection $p:\MM_{n,1}\to \R_n$.

\end{lemma}

\begin{proof} The homomorphism $t$ is induced from the map $T$ which is defined by replacing the marked point in $\MM_{n,1}$ by a circle $c_{n+1}$ of small radius centered at the point and lying in a plane parallel to the fixed one. 
For $(c,y)\in \MM_{n,1}$ let $T(c,y)\in\R_{n,1}$ be the configuration consisting of $c$
together with the circle $c_{n+1}$ of radius $\frac{1}{2}\operatorname{dist}(y,c)$ centered at $y$ and
lying in the plane through $y$ parallel to $P_0$. The element $T(c,y)$ lies indeed in $\R_{n,1}$ since the new circle
lies in a plane parallel to $P_0$, it is contained in a ball disjoint from $c$, and it is unlinked
from the other circles. As
$\operatorname{dist}(y,c)$ is continuous and strictly positive on $\MM_{n,1}$, the map $T$ is
continuous. Thus, it induces a homomorphism on fundamental groups, that is $t:\CM_{n,1}\to \LB_{n,1}$. Finally, by construction of $T$ the equality $q\circ t=p_*$ follows and this completes the proof.
\end{proof}

In the remaining part of this section we aim to prove the splitting of a short exact sequence which will be the main ingredient in proving Theorem \ref{thm1}. Before presenting this result we need the following two lemmas. We recall that the loop braid group and its pure subgroup, that is, $\LB_n$ and $\PLB_n$ can be seen as subgroups of $\Aut(\Fn)$, where their action on $\Fn$ has been described in Section \ref{pre}. Let $F_{n+1}=\langle x_1,\dots,x_{n+1}\rangle$. Recall that each $x_i$ corresponds to the meridian of the $i^{th}$ circle. For an element $u\in \Fn$ and for $1\leq i\leq n$, we define $\alpha_u, \alpha_{i,n+1}\in \PLB_{n+1}\subset\Aut(F_{n+1})$ by: $$\alpha_u(x_{n+1})=u^{-1}x_{n+1}u, \alpha_u(x_i)=x_i\ \text{and}\ \alpha_{i,n+1}(x_i)=x_{n+1}^{-1}x_ix_{n+1}, \alpha_{i,n+1}(x_j)=x_j\ \text{for all}\ j\neq i.$$
From a geometric point of view the element $\alpha_u\in \PLB_{n+1}$ is the pure loop braid where the last circle travels along $u$, while the other circles stay fixed, and the element $\alpha_{i,n+1}\in \PLB_{n+1}$ is the pure loop braid whose $i^{th}$ circle is passing through the last circle and returns back to its position, while the other circles stay fixed.

\begin{lemma}\label{map g}
The map $g:F_n\to \PLB_{n+1}\subset\Aut(F_{n+1})$ defined by $g(u)=\alpha_u$ is an injective homomorphism with image $U_{n+1}:=g(F_n)$ isomorphic to $F_{n}$. Moreover, for $\phi\in \LB_n\subset \Aut(F_n)$, seen as permutation-conjugacy automorphism of the free group, we have $\alpha_{\phi(u)}=\bar{\phi}\alpha_u\bar{\phi}^{-1}$, where $\bar{\phi}\in \LB_{n+1}\subset \Aut(F_{n+1})$ is the canonical extension of $\phi$, that is, $\bar{\phi}(x_i)=\phi(x_i)$ for $1\leq i \leq n$ and $\bar{\phi}(x_{n+1})=x_{n+1}$.
\end{lemma}

\begin{proof}
    For $u,v\in F_n$ one can verify that $\alpha_{uv}=\alpha_u\circ\alpha_v$. Moreover, $g$ is injective because if $\alpha_u=\alpha_v$ then $u^{-1}x_{n+1}u=v^{-1}x_{n+1}v$, which means that $vu^{-1}$ is in the centralizer of $x_{n+1}$ in $F_{n+1}$, which is the cyclic subgroup generated by $\langle x_{n+1}\rangle$ and thus $vu^{-1}\in \langle x_{n+1}\rangle\cap F_n=\{1\}$. Finally, $\bar{\phi}\alpha_u\bar{\phi}^{-1}(x_{n+1})=\bar{\phi}\alpha_u(x_{n+1})=\bar{\phi}(u^{-1}x_{n+1}u)=(\bar{\phi}(u))^{-1}x_{n+1}\bar{\phi}(u)$ and $\bar{\phi}\alpha_u\bar{\phi}^{-1}(x_i)=x_i$, for all $1\leq i\leq n$. Therefore, $\alpha_{\phi(u)}=\bar{\phi}\alpha_u\bar{\phi}^{-1}$. 
\end{proof}

Let us consider the locally trivial fibration of Theorem \ref{fibration} for $k=1$. In this case the fiber $F$ becomes $\Conf_1(M)=M$. We denote by $\iota_*$ the map from $\pi_1(M)\cong F_n$ to $\pi_1(\MM_{n,1})=\CM_{n,1}$, which is induced by the inclusion of the fiber $M$ into $\MM_{n,1}$, defined by $\iota_*(u)=a_u$, where $a_u$ is the element in $\CM_{n,1}$ whose circles stay fixed while the marked point traverses $u$. We denote by $U$ the image of $F_n$ under $\iota_*$, that is, $U:=\iota_*(F_n)$.

In the following lemma we aim to identify the element $a_u\in\CM_{n,1}$, whose marked point traces $u$ while the $n$ circles stay fixed, with the element $\alpha_u\in \PLB_{n+1}\subset\Aut(F_{n+1})$, seen as the loop braid where only the last circle moves and traces $u$. In particular, we want to determine a relation between $U$ and $U_{n+1}$. We recall that in Lemma \ref{thickening} we introduced the homomorphism $t:\CM_{n,1}\to \LB_{n,1}$, which we will use in what follows.

\begin{lemma}\label{correspondance}
Let $u\in\pi_1(M, *_1)\cong\Fn$ and let $a_u\in  \CM_{n,1}$. Then, $t(a_u)=\alpha_u\in \PLB_{n+1}\subset\Aut(F_{n+1})$. In particular, $U$ is isomorphic to $U_{n+1}$ via $t$. 

\end{lemma}

\begin{proof}
Under the identification of $\LB_n$ as a subgroup of $\Aut(F_{n+1})$ we have that $t(a_u)$ is the automorphism of $F_{n+1}$ induced by the terminal homeomorphism $h$ of an ambient isotopy realizing the "thickened" motion of $a_u$, where the marked point is replaced by a circle. Thus, it suffices to compute the images of the meridians $x_1,\dots,x_{n+1}$ of $C\sqcup c_{n+1}$ under $h$, each joined to a base-point $*$ in the collar of $\partial\BB^3$ by a fixed path $\tau$. 

 Let $u\in\pi_1(M, *_1)\cong\Fn$ be represented by the loop $\gamma$ in $M$ based at $*_1$. The element $t(\alpha_u)$ in $\PLB_{n+1}$ is the loop braid whose last circle follows $\gamma$, while the other circles $C$ stay fixed. Thus, it is realized by an ambient isotopy supported in a neighborhood $N$ of $\gamma$, containing $c_{n+1}$, which is disjoint form $C$. Hence, $h$ is the identity outside $N$.

 We represent the meridians $x_1,\dots,x_n$ by the loops $m_1,\dots,m_n$ based at $*$. The loops $\gamma$ and $m_1,\dots, m_n$ have different base-points and thus we can assume that they are disjoint and moreover we can assume that they are not inside $N$. Therefore, $x_1,\dots,x_n$ are fixed by $h$ and thus $t(a_u)(x_i)=x_i$, for all $1\leq i\leq n$.

 We represent the meridian $x_{n+1}$ by the loop $\tau m_{n+1}\bar{\tau}$, where $m_{n+1}$ is a small meridian of $c_{n+1}$. The isotopy returns $m_{n+1}$ to a meridian of $c_{n+1}$ but it drags the path $\tau$ along $\gamma$ and so $x_{n+1}$ is sent to $u^{-1}x_{n+1}u$. That is, $h(x_{n+1})=u^{-1}x_{n+1}u$. 

 Hence, we have $t(a_u)=\alpha_u$. Moreover, $a_u\in U, \alpha_u\in U_{n+1}$ and since by Lemma \ref{map g} the map that takes $u\in F_n$ to $\alpha_u\in U_{n+1}$ is an isomorphism, it follows that $U$ and $U_{n+1}$ are isomorphic via $t$.
\end{proof}

We are ready now to state and prove the result concerning the short exact sequence involving the locally trivial fibration proved in Theorem \ref{fibration}.

\begin{proposition}\label{ses}
    For every $k\geq 1$ we have the following short exact sequence:

    $$1\longrightarrow F_n^k\overset{\iota_*}{\longrightarrow}\CM_{n,k}\overset{p_*}{\longrightarrow}\LB_n\longrightarrow1.$$
    The short exact sequence splits, under the section $s_*$, and thus we have $\CM_{n,k}\cong F_n^k\rtimes\LB_n$, where $\LB_n$ acts on $F_n^k$ in the following way: for $\beta\in\LB_n$, whose induced element in $\Aut(F_n)$ is denoted by $\phi_\beta$, and for $(u_1,\dots,u_k)\in F_n^k$ we have $$s_*(\beta)^{-1}(u_1,\dots,u_k)s_*(\beta)=(\phi_\beta(u_1), \dots, \phi_\beta(u_k)).$$

    In particular, for $k=1$ we have $$\CM_{n,1}\cong U\rtimes\LB_n,\ \text{where}\ U\cong F_n.$$

\end{proposition}

\begin{proof}
    By Theorem \ref{fibration} we have the locally trivial fibration $p:\MM_{n,k}\to \R_n$ with fiber $\Conf_k(M)$, from which we deduce the following long exact sequence:
    $$\dots\longrightarrow\pi_2(\R_n)\overset{\partial}{\longrightarrow}\pi_1(\Conf_k(M))\overset{\iota_*}{\longrightarrow}\pi_1(\MM_{n,k})\overset{p_*}{\longrightarrow}\pi_1(\R_n)\longrightarrow\pi_0(\Conf_k(M))\longrightarrow\dots.$$
We know that $\pi_1(\Conf_k(M))\cong F_n^k$ by Lemma \ref{free}, $\pi_1(\MM_{n,k})=\CM_{n,k}$ and $\pi_1(\R_n)=\LB_n$. The fiber $\Conf_k(M)$ is path connected and thus $\pi_0(\Conf_k(M))=1$, which implies also that $p_*$ is surjective. Therefore, for the existence of the short exact sequence it remains to prove the injectivity of $\iota_*$, since the exactness at $\CM_{n,k}$ follows from the exactness of the long exact sequence. 

In order to prove injectivity of $\iota_*$ we will not make any claim about $\pi_2(\R_n)$, since $\R_n$ it is not aspherical because $\pi_1(\R_n)=\LB_n$ has torsion elements. Instead, we will simply prove injectivity of $\iota_*$ by proving injectivity of the map $t\circ\iota_*$, where $t$ is the map defined in Lemma \ref{thickening}. Let $k=1$ and $u,v\in F_n$. By Lemma \ref{correspondance} we have $t(\iota_*(u))=\alpha_u\in \LB_{n,1}\subset \Aut(F_{n+1})$ and $t(\iota_*(v))=\alpha_v\in \LB_{n,1}\subset \Aut(F_{n+1})$. If $\alpha_u=\alpha_v$ then $u^{-1}x_{n+1}u=v^{-1}x_{n+1}v$, which means that $vu^{-1}$ is in the centralizer of $x_{n+1}$ in $F_{n+1}$, which is the cyclic subgroup generated by $\langle x_{n+1}\rangle$ and thus $u=v$. Now let $k\geq 2$. For $1\leq i\leq k$ consider the map $q_i:\MM_{n,k}\to \MM_{n,1}$ that "forgets" all marked points but the $i^{th}$ one. This map is well defined and continuous because the marked points are ordered. Thus, it induces the map $(q_i)_*: \CM_{n,k}\to \CM_{n,1}$. Hence, we obtain $(q_i)_*\circ\iota_*=\iota_*^1\circ \mathrm{pr}_i$, where $\mathrm{pr}_i: F_n^k\to F_n$ is the projection to the $i^{th}$ coordinate and $\iota_*^1:F_n\to\CM_{n,1}$. Therefore, for $(u_1,\dots,u_k)\in \ker\iota_*$ then $\iota_*(u_i)=1$ for every $1\leq i\leq k$, which implies that $u_i=1$ by the case $k=1$. Thus, $\iota_*$ is injective.

Finally it follows that we indeed have the short exact sequence  $$1\longrightarrow F_n^k\overset{\iota_*}{\longrightarrow}\CM_{n,k}\overset{p_*}{\longrightarrow}\LB_n\longrightarrow1,$$
which splits since the map $s$ defined in Remark \ref{section} is a section of $p$. As a result, the homomorphism $p_*$ has as section the map $s_*:\LB_n\to\CM_{n,k}$, induced by $s$, as section and $\CM_{n,k}\cong \iota_*(F_n^k)\rtimes s_*(\LB_n),\ \text{where}\ \iota_*(F_n^k)\cong F_n^k\ \text{and}\ s_*(\LB_n)\cong \LB_n.$

It remains to show the action of $\LB_n$ on $F_n^k$. Let $\beta\in \LB_n$ be realized by $f$ with an isotopy $h_t$ from $\id$ to $f$ and let $\phi_\beta$ be the induced element in $\Aut(F_n)$. Moreover, we fix $i$ and we let $\alpha_u^i$ to be the fiber element in which the circles stay fixed, the $i^{th}$ marked point traverses a loop $l$ that represents $u$, while the other marked points are fixed. We consider $\Sigma(s,t):=(h_t(C),*_1,\dots,h_t(l(s)),\dots,*_k)$, which is an element in $\MM_{n,k}$, for all $s,t\in[0,1]$. Using the fact that $h_0=\id$, $h_1=f$ and $f$ fixes the marked points, we have for $s=0$ and $s=1$ the loop $s_*(\beta)$ and the loops $\alpha_u^i$, for $t=0$, and $\alpha_{\phi_\beta(u)}^i$, for $t=1$, which correspond to the fiber. Therefore we obtain 
\begin{equation}\label{eq}
    s_*(\beta)\alpha_{\phi_\beta(u)}^i=\alpha_{u}^is_*(\beta),
\end{equation}
simply by the fact that the boundary paths of the square from $(0,0)$ to $(1,1)$ are homotopic relative to their fixed endpoints. Thus, $s_*(\beta)^{-1}\alpha_u^is_*(\beta)=\alpha_{\phi(u)}^i$. The automorphism $\phi$ is the same for each $i$, and as a result we have that $s_*(\beta)^{-1}(u_1,\dots,u_k)s_*(\beta)=(\phi_\beta (u_1),\dots,\phi_\beta (u_k))$
\end{proof}

\section{Results on Nielsen fixed point classes}\label{main}

In this section we prove our two main results on Nielsen fixed point classes using all the theory we developed in Section \ref{mix}.
Throughout, let $f\in\Homeo(\BB^3, C)$ that fixes the boundary pointwise with $h_t$ an isotopy from $f$ to $\id$, from which we obtain $\beta\in \LB_n$ with induced $\phi_\beta\in\Aut(F_n)$. We want to compare fixed points of $f$ and so we set $k=1$ and we work in $\CM_{n,1}$, where $\CM_{n,1}\cong U\rtimes\LB_n$, as proved in Proposition \ref{ses}. 

We present below the notion of an extension of a loop braid $\beta\in \LB_n$ by an element $u\in F_n\cong\pi_1(M)\cong\iota_*(F_n)=U$. The element $u$ actually corresponds to a fixed point of $f$ and represents the path that it traces before it gets back to its position under the isotopy induced by $\beta$. Thus, we start with a loop braid $\beta\in \LB_n$ and we obtain an extension of it which is $\beta$ together with the loop that a fixed point traces in $M$. 

\begin{definition}
 For $u\in F_n$ we define the $u$-\textit{extension} of $\beta$ to be the element $\beta_u:=s_*(\beta)\alpha_u$ in $\CM_{n,1}$.
\end{definition}

In the following theorem we give a criterion for two elements $u,w \in F_n$ to be $\phi_\beta$-conjugate. That is, there exists an element $v\in F_n$ such that $w=\phi_\beta(v)uv^{-1}$.

\begin{theorem}\label{thm1}
    Let $u,w\in F_n$. Then, $u,w$ are $\phi_\beta$-conjugate if and only if $\beta_u$ and $\beta_w$ are conjugate by an element of $U$. 
\end{theorem}

\begin{proof} Let $u,v\in F_n$.
    From \eqref{eq} we have that $\alpha_vs_*(\beta)=s_*(\beta)\alpha_{\phi_\beta(v)}$ and by definition we have $\beta_u:=s_*(\beta)\alpha_u$. Thus, $\alpha_v\beta_u\alpha_v^{-1}=\alpha_vs_*(\beta)\alpha_u\alpha_v^{-1}=s_*(\beta)\alpha_{\phi_\beta(v)}\alpha_u\alpha_v^{-1}=s_*(\beta)\alpha_{\phi_\beta(v)}\alpha_{uv^{-1}}=s_*(\beta)\alpha_{\phi_\beta(v)uv^{-1}}$. The last two equalities follow from Lemma \ref{map g}. Therefore, $\alpha_v\beta_u\alpha_v^{-1}=\beta_{\phi_\beta(v)uv^{-1}}$. Hence, if $w$ and $u$ are $\phi_\beta$-conjugate, that is, there exists an element $v\in F_n$ such that $w=\phi_\beta(v)uv^{-1}$, then $\beta_w=\alpha_v\beta_u\alpha_v^{-1}$. Thus, $\beta_w$ and $\beta_u$ are conjugate by an element in $U$. Conversely, if $\beta_w=\alpha_v\beta_u\alpha_v^{-1}$ then $\beta_w=\beta_{\phi_\beta(v)uv^{-1}}$ and thus $w=\phi_\beta(v)uv^{-1}$.
\end{proof}

This theorem tells us that two fixed points of $f$ are "similar", in the Nielsen sense, if and only if the corresponding closed paths that they trace under the isotopy induce two different loop braid extensions of $\beta$ that are actually conjugate by an element in $U$.

We recall that $U\subset\CM_{n,1}$, which contains motions of the marked point, that $U_{n+1}\subset\LB_{n,1}\subset\Aut(F_{n+1})$, which contains motions of the last circle and that $U\cong U_{n+1}$ via the map $t$ that thickens the marked point to a circle.  

\begin{remark}
It is important to note that in the classical case of fixed points of homeomorphisms of the punctured disk in \cite{Guaschi}, Guaschi proves that the corresponding subgroup $U_{n+1}$ is the kernel of the map that forgets the last strand. In our case we can not replace the group $U\cong U_{n+1}$ by the kernel of the map that forgets the last circle, that is, $\LB_{n,1}\to \LB_n$ because the statement would be simply false. Let $q:\LB_{n,1}\to \LB_n\subset\Aut(F_{n})$ be the forgetful map, that forgets the last circle, which algebraically is induced by killing the $x_{n+1}$ generator. It holds that $U_{n+1}\subsetneq\ker q$, since for example $\alpha_{i,n+1}\in \ker q \setminus U_{n+1}$, for all $i\leq n$, which is the element whose $i^{th}$ circle passes through the last $(n+1)^{th}$ circle and goes back to its position. Indeed, $\alpha_{i,n+1}\notin U_{n+1}$ because every element of $U_{n+1}$ fixes $x_i$, for all $1\leq i\leq n$, whereas $\alpha_{i,n+1}(x_i)=x_{n+1}^{-1}x_ix_{n+1}\neq x_i$. It is also interesting to note that $t(\CM_{n,1})\subsetneq \LB_{n,1}$ and that the difference between $\LB_{n,1}$ and $t(\CM_{n,1})$ is that in $\LB_{n,1}$ we have elements where circles different than the last one pass through the last circle which can not happen for elements in $t(\CM_{n,1})$, since they are obtained from elements in $\CM_{n,1}$, where the last component is a point and not a circle. This is what forces Theorem \ref{thm1} to be stated with $U_{n+1}$ rather than $\ker q$.
\end{remark}

We will provide now a characterization of fixed points $x\in\Fix(f)\cap M$ that are Nielsen equivalent in $M$ from an algebraic point of view involving loop braid elements.

\begin{definition}
For $x\in\Fix(f)\cap M$ choose a path $\lambda_x$ from $*_1$ to $x$ in $M$ and let $\Lambda_x$ be the path $s\to(C,\lambda_x(s))$ in $\MM_{n,1}$ from $(C,*_1)$ to $(C,x)$. Let $\hat{l}_{x}$ be the mixed loop $(h_t(C),h_t(x))$ based at $(C,x)$ and set $l_{x}:=\Lambda_x\hat{l}_x\bar{\Lambda}_x\in \CM_{n,1}$. Thus, $l_x\in \CM_{n,1}$ is the motion of the circles together with the marked point $x$ based at $(C,*_1)$ under $h_t$. 
\end{definition}

The following lemma connects the extended loop braids $\beta_u\in \CM_{n,1}$ with the coordinate of a fixed point of $f$.

\begin{lemma}\label{coordinate}
    We have $l_x=\beta_{R(x)}$, where $R(x)\in F_n$ is the coordinate of $x$. Moreover, by choosing a different path from $*_1$ to $x$, that is, replacing $\lambda_x$ by $\gamma\lambda_x$, we are replacing $l_x$ by $\alpha_\gamma l_x\alpha_\gamma^{-1}$.
\end{lemma}

\begin{proof}
We will use the same argument as in the proof of Proposition \ref{ses}. Consider $$\Sigma(s,t):=(h_t(C), h_t(\lambda_x(s))),$$ which takes values in $\MM_{n,1}$, since $h_t$ is a homeomorphism and $\lambda_x(s)\notin C$. For $s=0$ and $s=1$ we get the loops $\hat{l}_{*_1}$ and $\hat{l}_x$ respectively, while for $t=0$ and $t=1$ we get the paths $\Lambda_x$ and $\Lambda_{f\circ \lambda_x}$ respectively, from $(C,*_1)$ to $(C,x)$. Thus, by homotopy relative to the endpoints of the square from $(0,0)$ to $(1,1)$ we get $\hat{l}_{*_1}\cdot \Lambda_{f\circ \lambda_x}=\Lambda_x\cdot \hat{l}_x$ and since $\hat{l}_{*_1}=s_*(\beta)$, we finally obtain $s_*(\beta)\cdot \Lambda_{f\circ \lambda_x}=\Lambda_x\cdot \hat{l}_x$. Therefore, $l_x=\Lambda_x\hat{l}_x\bar{\Lambda}_x=s_*(\beta)\cdot \Lambda_{f\circ \lambda_x}\bar{\Lambda}_x=s_*(\beta)\alpha_{(f\circ \lambda_x)\bar{\lambda}_x}=\beta_{R(x)}$. 

Moreover, by Lemma \ref{Nielsen}, we see that replacing $\lambda_x$ by $\gamma\lambda_x$, we replace $R(x)$ by $\phi_\beta(\gamma)R(x)\gamma^{-1}$, which in turn tells us, by Theorem \ref{thm1}, that we replace $\beta_{R(x)}$ by $\alpha_\gamma\beta_{R(x)}\alpha_\gamma^{-1}$.
    
\end{proof}

\begin{theorem}\label{thm2}
Let $x,y\in\Fix(f)\cap M$. The following are equivalent:
\begin{enumerate}
\item[(i)] the fixed points $x$ and $y$ are Nielsen equivalent in M, 
\item[(ii)] the coordinates $R(x)$ and R(y) are $\phi_\beta$-conjugate in $F_n$,
\item[(iii)] the loops $l_x$ and $l_y$ are conjugate in $\CM_{n,1}$ by an element of $U$,
\item[(iv)] the loop braids $t(l_x)$ and $t(l_y)$ are conjugate in $\LB_{n,1}$ by an element of $U_{n+1}$.
\end{enumerate}
  
\end{theorem}

\begin{proof}
   The equivalence between (i) and (ii) follows from Lemma \ref{Nielsen}, which is a well known fact. 
   
   The equivalence between (ii) and (iii) follows from Lemma \ref{coordinate} together with Theorem \ref{thm1}.
   
   Suppose that (iii) holds, that is, $l_y=\alpha l_x \alpha^{-1}$, with $\alpha\in U$. Then, $t(l_y)=t(\alpha)t(l_x)t(\alpha)^{-1}$, with $t(\alpha)\in U_{n+1}$, since $t$ is a homomorphism and $t(U)=U_{n+1}$, by Lemma \ref{thickening} and Lemma \ref{correspondance}. Thus, (iii) implies (iv). Conversely, suppose that (iv) holds. That is, $t(l_y)=\bar{\alpha}t(l_x)\bar{\alpha}^{-1}$, where $\bar{\alpha}\in U_{n+1}$. By Lemma \ref{correspondance}, there exists an element $\alpha\in U$ such that $t(\alpha)=\bar{\alpha}$. Thus, $t(l_y)=t(\alpha l_x\alpha^{-1})$ and it suffices to show that the homomorphism $t$ is injective. This is indeed the case and it follows from Lemma \ref{thickening} and Proposition \ref{ses}. If $t(m)=1$, for $m\in\CM_{n,1}$, then $p_*(m)=q(t(m))=1$, where $q:\LB_{n,1}\to \LB_n$ and so $m\in\ker p_*=U$. We know that the map $t$ is already injective on $U$ by Lemma \ref{correspondance}. Thus, $m=1$, which implies that $t$ is injective and therefore $l_y=\alpha l_x\alpha^{-1}$. Therefore, (iv) implies (iii). 
\end{proof}

\begin{remark}
Theorem \ref{thm2} asserts the existence of a conjugating element but we note that one can be written down explicitly. Indeed, let $\ell$ be a path from $x$ to $y$ with $\ell$ homotopic to $f\circ\ell$ relative to the endpoints,
realizing the Nielsen equivalence, and set $w:=\lambda_x\,\ell\,\bar\lambda_y\in\Fn$, where $\lambda_x$ is a path from $*_1$ to $x$ and $\lambda_y$ is a path from $*_1$ to $y$. Computing
the coordinate of $y$ with the path $\lambda_x\ell$ instead of $\lambda_y$ gives $R(x)$. Moreover, since the
two paths differ by $w$, Lemma \ref{Nielsen} gives $R(x)=\phi_\beta(w)R(y)w^{-1}$. Applying Theorem \ref{thm1} we get $\beta_{R(x)}=\alpha_w\beta_{R(y)}\alpha_{w}^{-1}$.
Since $l_x=\beta_{R(x)}$ and $l_y=\beta_{R(y)}$, by Lemma \ref{coordinate}, it follows that
$l_x=\alpha_w\,l_y\,\alpha_w^{-1}$.
Thus the conjugating element is built directly from the Nielsen path $\ell$ and the two connecting
paths.
\end{remark}

\begin{remark}
We notice that Theorems \ref{thm1} and \ref{thm2} apply equally to the iterates of $f$. The loop braid of $f^k$
becomes $\beta^k$ and its induced automorphism $\phi_\beta^k$. In particular, each Nielsen class of
$\Fix(f^k)$ in $M$ carries a loop braid, up to $U$-conjugacy, and the points of a
periodic orbit of period $k$ are fixed points of $f^k$ to which the criterion applies.
\end{remark}

     \section{Discussion}\label{application}
In this section we will present several remarks concerning our setting and our results and we will conclude with an open question.

\subsection{Realization and linking numbers of fixed points}

We note that Theorem \ref{thm2} would be of limited use if few classes were realized. In fact, as shown below, all classes are realized prescribing a motion of a fixed point by an ambient isotopy. 

\begin{proposition}\label{realization}
    Let $\beta\in \LB_n$ and let $u_1,\dots,u_m\in F_n$. There exists $f\in \Homeo(\BB^3, C)$ that is the identity on the collar of $\partial\BB^3$ with $b_{f,C}=\beta$ and with fixed points $x_1,\dots,x_m\in M$, such that $R(x_i)=u_i$ for every $1\leq i \leq m.$
\end{proposition}

\begin{proof}
    We consider the element $s_*(\beta)\alpha_{(u_1,\dots,u_m)}$ in $\CM_{n,m}\cong F_n^m\rtimes\LB_n$, where $\alpha_{(u_1,\dots,u_m)}$ is the element in $F_n^m$ that has the $i^{th}$ marked point traversing $u_i$ while the $n$ circles remain at their position. Note that $\alpha_{(u_1,\dots,u_m)}$ represents a loop $(c_t,y_1(t),\dots, y_m(t))$ based at $(C,*_1,\dots,*_m)$, which gives an isotopy of $C\cup\{*_1,\dots,*_m\}$ in $\Int(\BB^3)$. Thus, by the isotopy extension theorem it extends to an ambient isotopy $h_t$ of $\BB^3$, which is the identity on the collar of the boundary. We define $f:=h_1$ and thus $f(C)=c_1=C$, $b_{f,C}=\beta$ and $f(*_i)=y_i(1)=*_i$. Therefore, each $x_i:=*_i$ is a fixed point of $f$ in $M$. Moreover, the motion of $C$ together with $x_i$ determined by $h_t$ corresponds to the $u_i$-extension of $\beta$, that is, $\beta_{u_i}$. Thus, by Lemma \ref{coordinate}, $l_{x_i}=\beta_{u_i}$ and hence $R(x_i)=u_i$.
\end{proof}

In \cite{Makri}, the linking numbers of a fixed point $x$ with the components of $C$ are the data encoded by the abelianized generalized Lefschetz number. To be more precise, let $\mathrm{Ab}: F_n\to H_1(M)=\ZZ^n$ be the abelianization with $\mathrm{Ab}(x_i)=a_i$, then the coefficient of $a_i$ in $\mathrm{Ab}(R(x))$ equals $\lk(C_i,x)$. The exponents of the monomials that appear in the abelianized generalized Lefschetz number are obtained from the image of $\mathrm{Ab}(R(x))$ in $\mathrm{Coker}(f_*-\id)$, see Subsection $4.4$ in \cite{Makri}. 

\begin{remark}

By Theorem \ref{thm2} the $U$-conjugacy class of $l_x$ completely determines the Nielsen class
of $x$. The linking data extracted from the abelianized generalized Lefschetz number is not. That is, two
fixed points may have the same linking numbers with every sublink of $C$ and still lie in distinct
Nielsen classes, as the following example shows. 
\end{remark}

\begin{example}\label{example}
 Consider $\beta=1\in \LB_n$. Hence $\phi_\beta=\id$, which implies that the $\phi_\beta$-conjugacy is the standard conjugacy. Applying Proposition \ref{realization}, with $m=2$, $u_1=1\in F_n$ and $u_2=x_1x_2x_1^{-1}x_2^{-1}\in F_n$, we have a homeomorphism $f$ with two fixed points $x$ and $y$ such that $R(x)=1$ and $R(y)=x_1x_2x_1^{-1}x_2^{-1}$. Both coordinates $R(x), R(y)$ have trivial abelianization and thus all linking numbers of $x$ and $y$ vanish. While the fixed point $x,y$ have the same linking numbers we see that they are not Nielsen equivalent because $u_1=1$ and $u_2=x_1x_2x_1^{-1}x_2^{-1}\neq 1$ are not conjugate in $F_n$.
\end{example}

\subsection{Attaining the Nielsen number}
We know that the Nielsen number is a lower bound for the number of fixed points of any map in the homotopy class. We would like to know whether in our setting this lower bound is attained.

We recall that, given a map $f:X\to X$ of a compact
manifold $X$, the manifold $X$ is said to be \emph{Wecken} if $\MF(f)=\N(f)$ for all maps $f:X\to X$.

It is important to note that the present setting differs greatly from the two dimensional one. In \cite{Guaschi} the relevant space is $\mathbb{D}^2\setminus A$, where $A$ is an $n$-point set, and thus for $n\geq2$, we obtain a surface of negative Euler characteristic. We know from \cite{jiang1985} that such surfaces are non-Wecken, where Jiang proves that a surface is Wecken if and only if its Euler characteristic is non-negative. In dimension three no such obstruction occurs. In particular, we know that compact $3$-dimensional manifolds are Wecken. 

In our setting we work with homeomorphisms of the $3$-ball minus a trivial link of circles. Thus, applying Wecken's theorem \cite{Wecken}, the following results holds.

\begin{proposition} \label{wecken}
Let $\bar B$ be the compactification of $B=\BB^3\setminus C$ obtained by replacing each component of $C$ by the
boundary torus of a tubular neighborhood, and let $\bar f$ be the homeomorphism induced by $f$.
Then
$N(\bar f)=MF(\bar f).$
\end{proposition}

\begin{remark}
The trace formula, before passing to the abelianization, presented in \cite{Makri} computes the generalized Lefschetz number of $\bar f$, whose
nonzero coefficients correspond to the essential fixed point classes and hence give the Nielsen
number $N(\bar f)$. By Proposition \ref{wecken} this number is attained, which means that it is not just a lower
bound for the number of fixed points of a map homotopic to $\bar f$, but the exact minimum. This has
no analogue in the two-dimensional case. By Theorem~\ref{thm2}, each essential class is labeled by a $U_{n+1}$-conjugacy class of loop braids on $n+1$
strands.
\end{remark}

\subsection{Computability}
Let $G$ be a group and $\phi\in \Aut(G)$. We recall that two elements $u,v\in G$ are \emph{$\phi$-twisted} conjugate if there is an element $w\in G$ such that $v=\phi(w)uw^{-1}$. The $\phi$-twisted conjugacy problem in $G$ asks for an algorithm which, given $u,v\in G$ as finite words, decides whether $u$ and $v$ are $\phi$-twisted conjugate.

In our setting we work with automorphism $\phi_\beta$ in $F_n$. This automorphism, which is induced by $\beta\in\LB_n$, is given explicitly by the images
$\phi_\beta(x_i)$ of the generators of $\Fn$. By Theorem \ref{thm1} the elements $\beta_u$ and $\beta_v$ are $U$-conjugate if and only if $u$ and $v$ are $\phi_\beta$-conjugate. Combining our result with the result given in \cite{BMMV}, where they prove that the twisted conjugacy problem in finitely generated free groups is solvable, that is, there is an algorithm for the twisted conjugacy decision problem, the following result holds.

\begin{proposition}
    Given $\beta\in \LB_n$ and given $u,v\in F_n$ it is decidable whether $\beta_u$ and $\beta_v$ are $U$-conjugate. That is, it is solvable whether two fixed points with coordinates $u$ and $v$ are Nielsen equivalent. 
\end{proposition}

\subsection{Invariants of ribbon torus-links}
We know that loop braids close up to loop links, and due to Satoh's Tube map we know that loop links are sent to ribbon
torus-links in $S^4$ \cite{Satoh}. Moreover, we know that conjugate loop braids have equivalent closures \cite{Kamada}.
Since Nielsen equivalent fixed points have $U_{n+1}$-conjugate, in particular conjugate, loop
braids, then by Theorem \ref{thm2}, the ribbon torus-link obtained by closing $t(l_x)\in \LB_{n,1}$ is an invariant of the Nielsen class of $x$. Thus, every invariant of ribbon torus-links yields an invariant of Nielsen
classes. That is, if the closures of $t(l_x)$ and $t(l_y)$ are not isotopic then $x$
and $y$ are not Nielsen equivalent. This is of course weaker than Theorem~\ref{thm2} itself, which
sees $U_{n+1}$-conjugacy rather than conjugacy, but it makes the four-dimensional invariants
available.

\subsection{Question}
For surface homeomorphisms, Boyland \cite{Boyland} associated with each periodic orbit a braid type, and
introduced the braid forcing. That is, one braid type is said to force another if every
homeomorphism possessing an orbit of the first type must also have an orbit of the second type. The loop braid type would play the same role for homeomorphisms $f$ of $\BB^3$
leaving $C$ invariant, which raises the following question.

\begin{question}
Is there a notion of forcing on the loop braid types of periodic orbits of $f$, analogous to Boyland's
forcing for braid types of surface homeomorphisms?
\end{question}

\section*{Acknowledgement}

The author has received funding for this project from the European Union’s Horizon 2024 research and innovation programme under the Marie Skłodowska-Curie grant agreement No. $101205038$.

\bibliographystyle{plain}
\bibliography{bibliography}

\end{document}